\documentclass[12pt]{amsart}
\usepackage{hyperref}
\usepackage{color}

\usepackage{amsmath,amsfonts,amssymb}

\usepackage{graphics, pinlabel, color}
\usepackage{comment}

\usepackage[all,graph]{xy}
\usepackage{epstopdf}
\usepackage{tikz}
\usetikzlibrary{intersections}

\theoremstyle{plain}

\newtheorem{theorem}{Theorem}
\newtheorem{lemma}[theorem]{Lemma}

\newtheorem{question}[theorem]{Question}
\numberwithin{equation}{section}

\theoremstyle{remark}

\newtheorem*{ack}{Acknowledgements}

\theoremstyle{definition}

\def\Z{\mathbb{Z}}
\def\R{\mathbb{R}}

\def\CP2{\mathbb{CP}^2}

\def\be{\begin{equation}} \def\ee{\end{equation}}

\def\L{\L_{\R}}

\def\sm{\setminus}
\def\bp{\begin{pmatrix}}
\def\ep{\end{pmatrix}}
\def\bn{\begin{enumerate}}
\def\en{\end{enumerate}}

\def\ba{\begin{array}}
\def\ea{\end{array}}

\def\L{\Lambda}

\def\fr12{\frac{1}{2}}

\def\sp{\operatorname{sp}}

\begin{document}

\title{A specious unlinking strategy}

\author{Stefan Friedl}
\address{Fakult\"at f\"ur Mathematik\\ Universit\"at Regensburg\\   Germany}
\email{sfriedl@gmail.com}

\author{Matthias Nagel}
\address{Fakult\"at f\"ur Mathematik\\ Universit\"at Regensburg\\   Germany}
\email{matthias.nagel@mathematik.uni-regensburg.de}

\author{Mark Powell}
\address{
 D\'{e}partement de Math\'{e}matiques,
  UQAM, Montr\'{e}al, QC, Canada
 }
\email{mark@cirget.ca}


\def\subjclassname{\textup{2010} Mathematics Subject Classification}
\expandafter\let\csname subjclassname@1991\endcsname=\subjclassname
\expandafter\let\csname subjclassname@2000\endcsname=\subjclassname

\begin{abstract}
We show that the following unlinking strategy does not always yield an optimal sequence of crossing changes: first split the link with the minimal number of crossing changes, and then unknot the resulting components.
\end{abstract}
\maketitle



The \emph{unlinking number $u(L)$} of a link
 $L$ in $S^3$ is the minimal number of crossing changes required to turn a  diagram of $L$ into a diagram of the unlink. Here we take the  minimum over all diagrams of $L$.
Similarly, the  \emph{splitting  number $\sp(L)$}
is the minimal number of crossing changes required to turn a diagram of $L$ into a diagram of a split link. Once again the minimum is taken over all diagrams of $L$.\footnote{This definition was first introduced by Adams \cite{Ad96}. Unfortunately the term `splitting number' was used with a slightly different meaning in \cite{BS13,CFP13,BFP14}.}
Here recall that an  $m$-component link $L$ is   \emph{split} if there are $m$ disjoint balls  in $S^3$, each of which contains a component of $L$.

The detailed study of unlinking numbers and splitting numbers of links was initiated by Kohn \cite{Ko91,Ko93} in the early 1990s, and was continued by several other researchers, see e.g.\ \cite{Ad96,Ka96,Sh12,BS13,Ka13,CFP13}.  See also \cite{BW84, Tr88} for some early work.  Somewhat to our surprise, these are still relatively unstudied topics, and many basic questions remain unanswered.

In our investigations we wondered whether the computation of unlinking numbers can be separated into the two problems of splitting links and unknotting of knots.
 More precisely, the following question arose.

\begin{question}\label{question:efficient-strategy}
  Is the following always the most efficient strategy for unlinking? First split the link with the minimal number of crossing changes, and then unknot the resulting knots.
\end{question}

In this short note we will give a negative answer to the above question.
More precisely, we have the following theorem.

\begin{theorem}\label{thm:efficient-strategy}
The link $L$ in Figure~\ref{fig:our-fantastic-link} has the property
that any sequence of crossing changes for which the initial $\sp(L)$ crossing changes turn $L$ into a split link $S$, and where the remaining crossing changes unknot the components of $S$, has length greater than $u(L)$.
\end{theorem}

The remainder of this note comprises the demonstration of this theorem.
\begin{figure}[t]
    \labellist
    \small
    \pinlabel {$L_1$} at -20 130
    \pinlabel {$L_2$} at 80 295
    \pinlabel {$L_3$} at 415 295
    \endlabellist
    \includegraphics[scale=0.4]{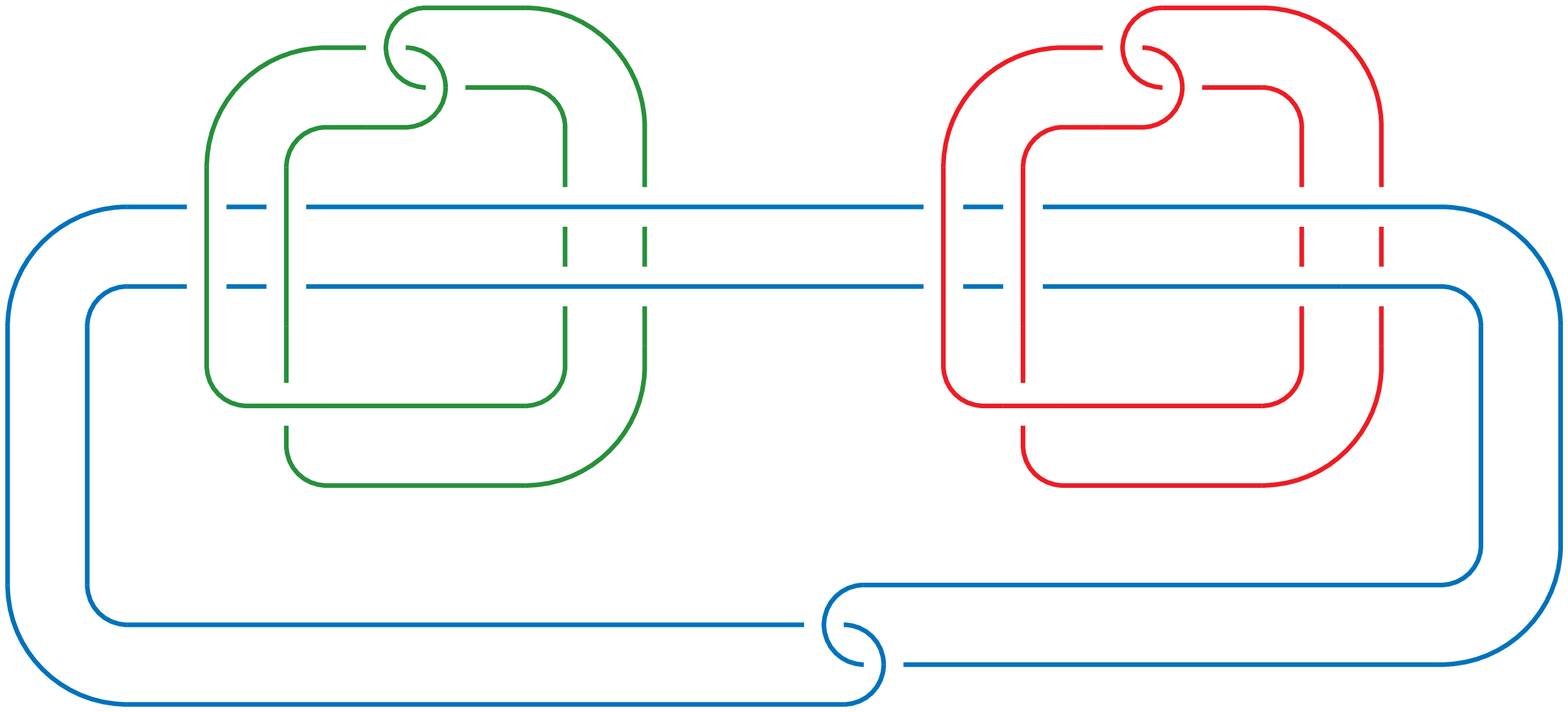}
    \caption{}
    \label{fig:our-fantastic-link}
  \end{figure}
First we check that the link is not already split.

\begin{lemma}\label{lemma:not-split}
  The sublinks $L_1 \cup L_2$ and $L_1 \cup L_3$ of $L$ are not split links.  In particular~$L$ is not a split link.
\end{lemma}

Before we provide  the proof of Lemma \ref{lemma:not-split} we introduce one more definition.
Given knots $K_1,\dots,K_m$ we denote the split link whose components
are $K_1,\dots,K_m$ by $K_1\sqcup \dots \sqcup K_m$.

\begin{proof}
First we claim that if $J$ is a 2-component split link $J_1 \sqcup J_2$, then any band sum of its components to form a knot $K$ will have determinant satisfying $\det(K)\cdot n^2 = \pm \det(J_1)\det(J_2)\cdot \ell^2$, for some nonzero integers $n,\ell \in \Z\sm\{0\}$.  To see this claim, note that by \cite{Mi98}, a knot arising from any band sum of a 2-component split link is concordant to the connected sum of the components.  Since $\det(K) = \Delta_K(-1)$, where $\Delta_K$ is the Alexander polynomial of $K$, the Fox-Milnor \cite{FM66} condition on Alexander polynomials of concordant knots\footnote{If $K$ and $J$ are concordant then there are Laurent polynomials  with integral coefficients $f$ and $g$ with $f(1) = \pm 1$ and $g(1)=\pm 1$ such that $\Delta_K(t) f(t)f(t^{-1}) \doteq \Delta_J(t) g(t)g(t^{-1})$.  Working modulo 2, we see that the conditions on $f$ and $g$ imply that $f(-1) \neq 0 \neq g(-1)$.}
 implies, by substituting $t=-1$, that $\det(K) \cdot n^2 = \pm \det(J_1 \# J_2) \cdot \ell^2$ for some nonzero integers $n,\ell \in \Z\sm\{0\}$.  Then the identity $\det(J_1 \# J_2)  = \pm \det(J_1)\det(J_2)$ completes the proof of the claim.

Now we use the claim to prove that $L_1 \cup L_2$ is not split.  The proof for $L_1 \cup L_3$ is identical since the two sublinks are isotopic: $L_1 \cup L_2 \cong L_1 \cup L_3$.
In our case, $J_1 = L_1$ is the unknot and $J_2= L_2$ is a trefoil.  Thus $\det(J_1)\det(J_2) = \Delta_{3_1}(-1) = (t^2-t+1)|_{t=-1} =3$, since the determinant of the unknot is one.

Apply  the Seifert algorithm to the diagram of the knot $L_2$ to obtain a genus one Seifert surface. There is a genus one Seifert surface for the knot $L_1$, disjoint from the Seifert surface for $L_2$, consisting of a long untwisted band and a $+1$ twisted `bridge' at the clasp.
 Add a band between the two which misses both Seifert surfaces.  This produces a knot $K$.
From the boundary connect sum of the two Seifert surfaces, using the band, we obtain a Seifert surface for $K$.  We compute a Seifert matrix for $K$ to be:
\[V = \begin{pmatrix}
  0 & 0 & 1 & 0 \\
  1 & 1 & 0 & 0 \\
  1 & 0 & -1 & 0 \\
  0 & 0 & 1 & -1
\end{pmatrix}.\]
From this we see that $\det(K) = \det(V+V^T) = 13$.  There do not exist any nonzero integers $n,\ell$ such that $13 \cdot n^2 = \pm 3 \cdot \ell^2$, by the uniqueness of prime factorisations, in which the parities of the exponents of $3$ and $13$ will never be equal. We deduce that $L_1 \cup L_2$ cannot be a split link, as desired.
\end{proof}

Now we use Lemma~\ref{lemma:not-split} to prove that the link $L$ of Figure~\ref{fig:our-fantastic-link} indeed has the properties claimed in
Theorem~\ref{thm:efficient-strategy}.

\begin{proof}[Proof of Theorem \ref{thm:efficient-strategy}]
The component labelled $L_1$ is an unknot, while the components $L_2$ and $L_3$ are trefoils.  Observe that a single crossing change on $L_1$, undoing the clasp, yields a split link $L_1 \sqcup L_2 \sqcup L_3$.  Therefore, since the splitting number is nonzero by Lemma~\ref{lemma:not-split}, the splitting number of $L$ is one: $\sp(L)=1$.

Since the unknotting number of the trefoil is one, we require at least one $(L_2,L_2)$ crossing change and at least one $(L_3,L_3)$ crossing change to turn $L$ into the unlink.  Therefore the unlinking number of $L$ is at least two.  Observe that the unlinking number is exactly two, since we may undo the clasps of $L_2$ and of $L_3$ and thus obtain the unlink: $u(L) =2$.

We need to see that any unlinking sequence which begins by splitting the link with a single crossing change must include at least two further crossing changes, making a total of three changes.  For then splitting before unknotting will be less efficient than an optimal crossing change sequence for unlinking.

We claim that splitting $L$ with one crossing change is only possible with an $(L_1,L_1)$ crossing change.  To see the claim, first note that a single crossing change between different components changes the corresponding pairwise linking number, which is originally zero.  Links with nonzero linking number cannot be split.  Therefore we need to consider a single $(L_2,L_2)$ change and a single $(L_3,L_3)$ change.  However an $(L_2,L_2)$ crossing change does not alter the link type of $L_1 \cup L_3$, which by Lemma~\ref{lemma:not-split} is not split.  Similarly an $(L_3,L_3)$ crossing change does not alter the link type of $L_1 \cup L_2$, which  again by Lemma~\ref{lemma:not-split} is not split.  The only remaining case is that of an $(L_1,L_1)$ change, which proves the claim.

However any $(L_1,L_1)$ crossing change does not alter the knot type of $L_2$, nor that of $L_3$.  Any split link resulting from one crossing change on $L$ is the split union of two trefoils and another knot (almost certainly the unknot, but we do not need this nor do we claim it.)
The two trefoils require one further crossing change each to unknot them.  This completes the proof
of Theorem \ref{thm:efficient-strategy}.
\end{proof}


We conclude this paper with the following question.

\begin{question}
Is there an example where all the components begin as unknots?
\end{question}

Note that even when all the components of a link are unknots, the optimal splitting can still necessarily produce knots~\cite{Ad96,BFP14}.

\begin{ack}
We thank Patrick Orson for providing us with Question~\ref{question:efficient-strategy}.
All three authors also  gratefully acknowledge the support provided by the SFB 1085 `Higher
Invariants' at the University of Regensburg, funded by the Deutsche
For\-schungsgemeinschaft (DFG).
\end{ack}

\end{document}